\documentclass{amsart}
\usepackage{amssymb,amsmath,amsthm,mathrsfs, array}
\usepackage{amsfonts,latexsym,amsxtra,amscd}
\usepackage{pgfplots, tikz}
\usepackage{graphicx, stmaryrd}

\usepackage{fancyhdr}
\usepackage{float}

\newtheorem{theorem}{Theorem}[section]
\newtheorem{lemma}[theorem]{Lemma}

\theoremstyle{definition}
\newtheorem{definition}[theorem]{Definition}

\newtheorem{corollary}[theorem]{Corollary}
\newtheorem{proposition}[theorem]{Proposition}
\theoremstyle{remark}
\newtheorem{remark}[theorem]{Remark}
\numberwithin{equation}{section}

\begin{document}

\title{Chaotic and topological properties of  continued fractions }

\author{ Weibin Liu}
\address{Guangdong University of Foreign Studies South China Business College, Guangzhou, 510000, P. R. China}
\email{weibinliu@whu.edu.cn}
\thanks{*Corresponding author.}

\author{Bing Li$^*$}
\address{Department of Mathematics, South China University of Technology, Guangzhou, 510640, P. R. China}
\email{scbingli@scut.edu.cn}


\date{\today}


\begin{abstract}
We prove that there exists a scrambled set for the Gauss map with full Hausdorff dimension.
Meanwhile, we also  investigate the topological properties of the sets of points with dense or non-dense orbits.
\end{abstract}

\maketitle
\section{Introduction}

It is known that every irrational number $x\in [0,1)$ can admit an infinite continued fraction (CF)
induced by  Gauss transformation $T :[0,1)\rightarrow [0,1)$ given by
\begin{center}
$T(x):=\dfrac{1}{x}-\left\lfloor\dfrac{1}{x}\right\rfloor$~~for~$x\in(0,1)$~~and ~~$T(0):=0 $
\end{center}
where $\lfloor y \rfloor$ denotes the integer part of a real number  $y$, that is,
 $\lfloor y \rfloor=n$, if $y\in[n,n+1)$ for some $n\in \mathbb{Z}$. We set
 $$ a_1(x):=\lfloor x^{-1}\rfloor \quad \text{and} \quad a_n(x):=\lfloor (T^{n-1}(x))^{-1}\rfloor,~~ n\geq 2,$$
and have the following CF expansion of $x$:
\begin{equation*}
x=\dfrac{1}{a_1(x)+\dfrac{1}{a_2(x)+\dfrac{1}{a_3(x)+\cdots}}}=[a_1(x),a_2(x),a_3(x),\cdots].
\end{equation*}
The numbers $a_n(x)$ ($n\geq 1$) are called the partial quotients of  $x$.
In 1845, Gauss observed that $T$ preserves the probability measure given by
$$\mu(B)=\dfrac{1}{\log 2}\int_B\dfrac{1}{1+x}d\mathcal{L}(x),$$
where $B\subset [0,1)$ is any Borel measurable set and
$\mathcal{L}$ is the Lebesgue measure on $[0,1]$. The measure $\mu$ is called Gauss measure and equivalent with $\mathcal{L}$.

Continued fractions is a kind of representation of real numbers and an important tool to study the Diophantine approximation in number theory. Many metric and dimensional results on Diophantine approximation have been obtained with the help of continued fractions such as Good \cite{G}, Jarnik \cite{J}, Khintchine \cite{K} etc, and also the dimensional properties of continued fractions were considered, for example Wang and Wu \cite{WangWu08},  Xu \cite{Xu} etc. The behaviors of the continued fraction dynamical systems have been widely investigated, for example, the shrinking target problems \cite{LWWX14}, mixing property \cite{Philipp70}, thermodynamic formalism \cite{Mayer90}, limit theorems \cite{Faivre98} etc.  Hu and Yu \cite{HY}  were concerned with  the set
$\{x\in [0,1): \xi\not\in \overline{\{T^{n}(x): n\geq 0\}}\}$ for any $\xi\in [0,1)$
and proved that it  is 1/2-winning and thus it has full Hausdorff dimension.

 Asymptotic behaviors of the orbits are one of the main theme in dynamical systems. In this paper, we focus on the continued fraction dynamical system $([0, 1), T)$ and study firstly the denseness of the orbits, and secondly the size of the scrambled set from the sense of Hausdorff dimension.
Let
\begin{center}
$D:=\{x\in[0,1):$ the orbit of $x$ under $T$ is dense in  $[0,1]~~\},$
\end{center}
and $D^{c}$ be the complement of the set $D$ in $[0,1]$.
\begin{theorem}\label{thm1}
Let $T$ be the Gauss map on $[0,1)$. Then\\
  $(\rm 1)$ $D$ is of full Lebesgue measure in $[0,1]$.\\
  $(\rm 2)$ $D^{c}$ is of full Hausdorff dimension.\\
 $ (\rm 3)$ $D^{c}$ is of the first category in $[0,1]$.\\
 $(\rm 4)$ $D^{c}$ and $D$ are dense in $[0,1]$.
\end{theorem}

The second statement $(\rm 2)$ of the theorem is a direct corollary of the result in \cite{HY} since the set studied in \cite{HY} is a subset of $D^{c}$.
Similar properties with this theorem for $\beta$-transformations have been given by \cite{LC}. For some other dynamical systems, the set of the points with nondense orbits is usually of full Hausdorff dimension, see also \cite{Kleinbock98}, \cite{Urbanski91} and references therein.

The chaotic property is a characterization of the asymptotic behaviors between different orbits in dynamical system.

\begin{definition}\cite{LY}
Let $(X,\rho)$ be a metric space. For two points $x,y\in X$, $(x,y)$ is a scrambled pair for
the map $f:X\rightarrow X$, if
\begin{center}
$\limsup\limits_{n\rightarrow \infty}\rho(f^{n}(x),f^{n}(y))>0$ and
$\liminf\limits_{n\rightarrow \infty}\rho(f^{n}(x),f^{n}(y))=0$.
\end{center}
A subset $S\subseteq X$, containing at least two points, is a scrambled set of $f$,
if for any $x,y\in S$, $x\neq y$,  $(x,y)$ is a scrambled pair for $f$.
If  a scrambled set $S$ for $f$ is  uncountable, we say that $f$ is chaotic in the sense of
Li-Yorke.
\end{definition}

It is  well known that \cite{BGK} the surjective continuous transformation on compact metric space with  positive topological entropy is
 chaotic in the sense of Li-Yorke. Even though  the topological entropy of the  Gauss map is infinite, the result above in \cite{BGK} can not be applied since the Gauss map is not continuous. We prove that the  Gauss map is chaotic as the following, indeed, the scrambled set can be very large from the dimensional sense, not just uncountable.

\begin{theorem}\label{thm2}
Let $T$ be the Gauss map on $[0,1)$. Then there exists a scrambled set in $[0,1)$ whose Hausdorff dimension is 1.
\end{theorem}

As a consequence, we have
\begin{corollary}
The Gauss transformation on $[0,1)$ is chaotic in the sense of Li-Yorke.
\end{corollary}

The sizes of the scrambled sets have been considered for many dynamical systems from the sense of the measure, dimension, and topology, see also \cite{BL10}, \cite{Mizera88}, \cite{Neu10},  \cite{WangXiong05}, \cite{WangXiong06} etc.  In 1995, Xiong \cite{X} proved that there exists a scrambled set of $\{1,2,\cdots,N\}^{\mathbb{N}}$ of full Hausdorff dimension .

The paper is organized as follows. In Section 2, we collect and establish some elementary
properties of continued fractions which will be used later. Section 3 is devoted to proving
Theorem \ref{thm1} and Theorem \ref{thm2} is proved in Section 4.

\section{Preliminaries}

In this section, we present some elementary results in the theory
of  symbolic dynamics and continued fractions which will be used later,see \cite{K}.

\subsection{Basic concepts}

Let $\mathcal{A}=\{1,2,\cdots,N\}$ with $N\geq 2$ or $\mathcal{A}=\mathbb{N}=\{1,2,\cdots,n,\cdots\}$.
Denote  the symbolic space of  one-sided infinite sequences over $\mathcal{A}$ by
 $$\mathcal{A}^{\mathbb{N}}=\{x=(x_{1},x_{2},\cdots) :\quad x_{i}\in \mathcal{A},~\forall i\in
\mathbb{N}\}.$$
The symbol $x_i$ is called the $i$-th coordinate of $x$.

 We assign the discrete topology to $\mathcal{A}$ and the product topology to  $\mathcal{A}^{\mathbb{N}}$.
For $x,y\in \mathcal{A}^{\mathbb{N}}$, the distance $d$ is defined  by
\begin{center}
$d(x,y)=2^{-i}$,  where $i= \inf\{j\geq 0: x_{j+1}\neq y_{j+1}\}.$~
\end{center}
Denote by  $\mathcal{A}^{n}$  the set of all $n$-letter words and
$\mathcal{A}^{*}$ for the set of all  finite words over $\mathcal{A}$.
For any $u$, $v\in \mathcal{A}^{*}$, $uv$ denotes the concatenation of $u$ and $v$.
The symbol ``$|\cdot|$" means the diameter, the length and the absolute value with respect to
a set, a word, and a real number respectively.
For $i,j\in \mathbb{N}$ with $i< j$, we write $x|^{j}_{i}=x_{i},x_{i+1},\cdots, x_{j}$.~

 The shift map $\sigma:\mathcal{A}^{\mathbb{N}}\rightarrow\mathcal{A}^{\mathbb{N}} $
 is defined by
 $$(\sigma(x))_{i}=x_{i+1},~\forall~ i\in \mathbb{N}.$$

\subsection{Continued fractions}

In this subsection,  we collect and establish some elementary properties on
continued fractions. We can refer to \cite{K} for more details.

We denote by $p_n(x)/q_n(x) := [a_1(x),a_2(x),\cdots,a_n(x)]$ the $n$-th convergent of $x$. With the
conventions $p_{-1}(x)=1, q_{-1}(x)=0, p_0(x)=0, q_0(x)=1$, we have
$$ p_{n+1}(x)=a_{n+1}(x)p_n(x)+p_{n-1}(x),~~ n\geq0,$$
$$q_{n+1}(x)=a_{n+1}(x)q_n(x)+q_{n-1}(x),~~ n\geq0. $$

For any $n\geq 1$ and $(a_1,\cdots,a_n)\in \mathbb{N}^n$, let $q_n(a_1,\cdots,a_n)$ be the denominator of
finite continued fraction $[a_1,\cdots,a_n]$. If there is no confusion, we write $q_n$ instead of
$q_n(a_1,\cdots,a_n)$ for simplicity.
For any $n\geq 1$ and $(a_1,\cdots,a_n)\in \mathbb{N}^n$,
we define the $n$-th basic intervals $I(a_1,\cdots,a_n)$ by
$$I(a_1,\cdots,a_n)=\{x\in[0,1): a_i(x)=a_i,  1\leq i\leq n\}. $$
 In the case of $n=1$, $I(i)=\left[\dfrac{1}{i+1},\dfrac{1}{i}\right)$ for $i\geq 1$.
The following lemmata are well known.
\begin{lemma}\cite{K}\label{lem0}
For any $n\geq 1$ and $(a_1,\cdots,a_n)\in \mathbb{N}^n$, we have:\\
$(\rm i)$ $q_n(a_1,\cdots,a_n)\geq 2^{\frac{n-1}{2}}$;\\
$(\rm ii)$ $|I(a_1,\cdots,a_n)|=\dfrac{1}{q_n(q_n+q_{n-1})}.$
\end{lemma}

\begin{lemma}\cite{W} \label{lem01}
For any $n\geq 1$ and $1\leq k\leq n$,
$$\dfrac{a_k+1}{2}\leq \dfrac{q_n(a_1,\cdots,a_n)}{q_{n-1}(a_1,\cdots,a_{k-1},a_{k+1},\cdots,a_n)} \leq a_k+1.$$
\end{lemma}

For any $N\in \mathbb{N}$, let
$$E_N=\{x\in[0,1):~~1\leq a_n(x)\leq N, ~~\text{for ~any} ~~n\geq 1 \},$$
and
$$E=\{x\in[0,1):\sup\limits_{n\geq 1}a_n(x)<+\infty\}=\bigcup\limits_{N=1}^{\infty}E_N.$$
Jarnik \cite{J} showed the following:
\begin{lemma}\label{lem1}
For any  $N\geq 8$,
$$1-\dfrac{1}{N\log2}\leq \dim_HE_N \leq 1-\dfrac{1}{8N\log N},$$
and $$\dim_HE=1,$$
where ``~~$\dim_H$" denotes the Hausdorff dimension.
\end{lemma}

\begin{lemma}\cite[Lemma 4.1]{HY}\label{lem4}
Let  $T:[0,1)\rightarrow [0,1)$ be the Gauss transformation. Then
 there exists a constant $\lambda \geq 1$ such that for any $u,v\in \mathbb{N}^{*}$,
$$\lambda^{-1}|I(u)||I(v)|\leq|I(uv)| \leq \lambda|I(u)||I(v)|.$$
\end{lemma}

\section{The proof of Theorem \ref{thm1}}

\begin{lemma}\label{lem2}
Let $[a_1(x),a_2(x),a_3(x),\cdots]$ be the CF expansion of $x\in  [0,1)$. Then $x\in D^{c}$ if and only if
there exists some word  $u\in \mathbb{N}^{*}$
which does not appear in $(a_1(x),a_2(x),a_3(x),\cdots)$.
\end{lemma}
\begin{proof}
On the one hand, if there exists  $u\in \mathbb{N}^{*}$ which does not appear in $(a_1(x),a_2(x),a_3(x),\cdots)$,
let $y$ be the middle point of the interval $I(u)$, then $|y-T^{n}(x)|\geq |I(u)|/2$ for any $n\geq 0$.
Hence $x\in D^{c}$ by the definition of the set $D^{c}$.

On the other hand,  suppose that  $x\in D^{c}$.
When $x\in \mathbb{Q}\cap[0,1)$, since the CF expansion of every rational
number is finite, then there exists $k=k(x)\geq 1$ such that $x=[a_1(x),a_2(x),\cdots, a_k(x)]$,
thus the finite word  $u:=(a_1(x),a_2(x),\cdots, a_k(x),1)\in\mathbb{N}^{k+1}$ is the required.
Now it left to check the case of $x\in \mathbb{Q}^{c}\cap[0,1)$.

By contradiction, suppose that every  $u\in \mathbb{N}^{*}$ appears in  $(a_1(x),a_2(x),a_3(x),\cdots)$, then, for any
$y\in \mathbb{Q}^{c}\cap[0,1)$ and any $n\geq 1$, the finite word
$(a_1(y),a_2(y),\cdots, a_n(y))\in\mathbb{N}^{n}$ will
 appear in $(a_1(x),a_2(x),\cdots)$.
 Hence there exists $m=m(y,n)$ such that $T^{m}(x)\in I(a_1(y),a_2(y),\cdots, a_n(y))$,
 then we obtain
$$|y-T^{m}(x)|\leq |I(a_1(y),a_2(y),\cdots, a_n(y))|\leq q_n(y)^{-2}\leq 2^{-(n-1)}$$
 by $(\rm i)$ of Lemma \ref{lem0}, it follows that
 $$\mathbb{Q}^{c}\cap[0,1) \subset \overline{\{T^{k}(x): k\geq 0\}},$$
 therefore
 $$[0,1]= \overline{\mathbb{Q}^{c}\cap[0,1)} \subset  \overline{\{T^{k}(x): k\geq 0\}}. $$
that is to say,  $x\in D$ by the definition of the set $D$, which contradicts $x\in D^{c}$. Then the lemma holds.
\end{proof}

\noindent{\bf Proof of Theorem \ref{thm1}:}

(1)  In view of Lemma \ref{lem2}, we have
\begin{equation}\label{Eq.1}
D^{c}=\bigcup\limits_{k\geq 1}\bigcup\limits_{u\in \mathbb{N}^{k}}\big\{x\in [0,1): u~~
\text{ does ~~not~~ appear~~ in}~~
(a_1(x),a_2(x),a_3(x),\cdots) \big\}:=\bigcup\limits_{k\geq 1}\bigcup\limits_{u\in \mathbb{N}^{k}}F_{u}.
\end{equation}
Since the cardinality of $\mathbb{N}^{k}$ is countable, we just need to prove $\mathcal{L}(F_u)=0$,
where $\mathcal{L}$ is the Lebesgue measure on $[0,1]$. The fact that $T$ is ergodic with respect to the Gauss
measure $\mu$ is equivalent to that
$$\mu\Big(\bigcup\limits_{n\geq 0}T^{-n}(B)\Big)=1$$
 for any Borel measurable  subset
$B\subseteq [0,1)$ of positive measure. Since  $\mu(I(u))>0$, we have
 $$\mu\Big(\bigcup\limits_{n\geq 0}T^{-n}(I(u))\Big)=1.$$
Since
 $\bigcup\limits_{n\geq 0}T^{-n}(I(u))\cap F_u=\emptyset$, the assertion
 $\mu(F_u)=0$ follows, which implies $\mathcal{L}(F_u)=0$
 by the equivalence between the Lebesgue and Gauss measures.

(2) Since $E_N\subseteq D^{c}$ for any $N\geq 1$ by Lemma \ref{lem2}, we have $E\subseteq  D^{c}\subseteq [0,1)$. Hence
$\dim_{H} D^{c}=\dim_{H}E=1$ by Lemma \ref{lem1}.

(3) According to (\ref{Eq.1}),
we only need to prove that $F_u$ is nowhere dense since $\mathbb{N}^{k}$ is countable.
Recall that the Gauss measure $\mu$ is an ergodic measure with respect to $T$.
Let $b$ and $c$ be the endpoints of the interval $I(u)$ and $b<c$.
Since  the open interval $(b,c)\subset I(u)$,  we have
 $$\mu\Big(\bigcup\limits_{n\geq 0}T^{-n}(b,c)\Big)=1, ~~ \text{and} ~~
 \bigcup\limits_{n\geq 0}T^{-n}(b,c)\cap F_u \subset \bigcup\limits_{n\geq 0}T^{-n}(I(u))\cap F_u =\emptyset.$$
Then
$ F_u \subset \Big(\bigcup\limits_{n\geq 0}T^{-n}(b,c)\Big)^{c}. $
Let  $\overline{F_u}$ be the closure of $F_u$.
Since  $\bigcup\limits_{n\geq 0}T^{-n}(b,c)$ is an open set,
thus
$$  \overline{F_u} \subset \Big(\bigcup\limits_{n\geq 0}T^{-n}(b,c)\Big)^{c}~~\text{and}~~
\mu(\overline{F_u})\leq \mu\Big \{\Big(\bigcup\limits_{n\geq 0}T^{-n}(b,c)\Big)^{c} \Big\}=0.$$
Then $\mathcal{L}(\overline{F}_u)=0$ by the equivalence between the Lebesgue and Gauss measures.
Hence $F_u$ is nowhere dense.

(4) Since $\mathcal{L}(D)=1$ by (1), we know that $D$ is dense in $[0,1]$. It suffices
to prove that  $D^{c}$ is dense in $[0,1]$, which can be obtained by the facts
that $E\subseteq D^{c}$ and that the set $E$ is dense in $[0,1]$.
\qed

\section{The proof of Theorem \ref{thm2}}

The idea of  proving
Theorem \ref{thm2} is to construct a scrambled set in $\mathbb{N}^\mathbb{N}$, then project such set to the unit interval $[0, 1)$, and finally show the projection is of full Hausdorff dimension.
%
%

\subsection{Construction of a scrambled set}
Fix $N\geq 2$, we define four  mappings $g_N$,  $\Theta_N$, $\Psi_{N}$ and $\Delta_N$ which
are dependent on the symbol $N$. Actually, the idea for constructing  these mappings
 is inspired by Xiong \cite{X}.  In the following, we write $\Sigma_N=\{1,2,\cdots,N\}^{\mathbb{N}}$ to emphasize the dependence of the number $N$ of letters.

(1). Define a map $g_N:\Sigma_N \rightarrow \Sigma_N$ by
\begin{equation*}
(g_N(x))_n=
\begin{cases}
N & \text{if  $n=1$,}\\
1  & \text{if  $2+k(k-1)\leq n\leq 1+k^2$ for $k\geq 1$,}\\
x_{n-1-k^2} & \text{if  $2+k^2\leq n\leq 1+k^2+k$ for $k\geq 1$,}
\end{cases}
\end{equation*}
for  $x=(x_1,x_2,x_3,\cdots)\in \Sigma_N$. That is to say, the form of $g_N(x)$ is as follows:
$$g_N(x)=(N,1,x_1, 1,1,x_1,x_2, 1,1,1,x_1,x_2,x_3, \cdots ,\underbrace{1,1,\cdots, 1}_{n},
x_1,x_2, \cdots ,x_n, \cdots)$$

(2). Define  $\Theta_N:\Sigma_N \rightarrow \Sigma_N $ by
$$\Theta_N(x)=(x_1,x_1,x_2,x_1,x_2,x_3,\cdots, x_1,x_2, \cdots ,x_n,x_1,x_2, \cdots ,x_{n+1},\cdots)$$
for  $x=(x_1,x_2,x_3,\cdots)\in \Sigma_N$.

(3). Define a positive integers sequence
$R=\{r_n\}_{n\geq 1}:=\bigcup\limits_{m=1}^{\infty}\bigcup\limits_{t=1}^{m}\{k:k=m^3+t\}$,
that is,  $R=\{2,9,10, 28,29,30, \cdots\}$.
Define  $\Psi_{N}:\Sigma_N\times \Sigma_N \rightarrow \Sigma_N$ by
\begin{equation*}
(\Psi_{N}(x,y))_n=
\begin{cases}
y_k & \text{if there exists $k\geq 1$ such that  $n=r_k$,}\\
x_{n-k+1} & \text{if there exists $k\geq 1$ such that $r_{k-1}<n<r_k$,}
\end{cases}
\end{equation*}
for  $x=(x_1,x_2,\cdots)\in\Sigma_N$ and $y=(y_1,y_2,\cdots)\in \Sigma_N$.

(4).Define  $\Delta_N:\Sigma_N\rightarrow \Sigma_N$ such that for
$x\in \Sigma_N$,
$$\Delta_N(x)=\Psi_{N}(x,\Theta_N\circ  g_N(x) ).$$
That is to say, the form of $\Delta_N(x)$ is as follows:
\begin{eqnarray*}
\begin{split}
\Delta_N(x)
&=\left(x_1,(\Theta_N(g_N(x)))_1,~~x|_2^7,~~(\Theta_N(g_N(x)))|_2^3,~x|_8^{24},~~
(\Theta_N(g_N(x)))|_4^6,~~x|_{25}^{58},~~
(\Theta_N(g_N(x)))|_7^{10}, \cdots \right)\\
&=\left(x_1,(g_N(x))_1,~~x|_2^7,~~(g_N(x))|_1^2,~x|_8^{24},~~
(g_N(x))|_1^3,~~x|_{25}^{58},~~
(g_N(x))|_1^4, \cdots \right)
\end{split}
\end{eqnarray*}

\begin{remark}\label{rem4.1}
(1). The mappings $g_N$, $\Theta_N$ and  $\Psi_{N}$    are continuous and injective.
The mapping $\Delta_N$ is a continuous bijection from $\Sigma_N$ to $\Delta_N(\Sigma_N)$.

(2).$(\Theta_N(g_N(x)))_1=(\Theta_N(g_N(x)))_2=(g_N(x))_1=N$ and $(\Theta_N(g_N(x)))_3=(g_N(x))_2=1$ for all $x\in \Sigma_N$.
\end{remark}

Let  $S_N:=\Delta_N(\Sigma_N)$ for any $N\geq 2$ and $S:=\bigcup\limits_{N\geq 2}S_N$.
The following lemmata  indicate $S$ is a scrambled set of the shift on $\mathbb{N}^{\mathbb{N}}$.
\begin{lemma}
For any $M,N\geq 2$ with $M\neq N$, then $S_M\cap S_N=\emptyset$.
\end{lemma}
\begin{proof}
Suppose that $M>N\geq 2$. Note that $$ S_M=\Delta_M(\Sigma_M)=\Delta_M(\{1,2,\cdots,M\}^{\mathbb{N}}),$$
and  $$S_N=\Delta_N(\Sigma_N)=\Delta_N(\{1,2,\cdots,N\}^{\mathbb{N}}).$$
By the definitions of $\Delta_M$ and $\Delta_N$ and $M>N$, the symbol $M$ appears infinitely often
 in any $y\in S_M$,
but it  does not appear in any $x\in S_N$. Hence $S_M\cap S_N=\emptyset$.
\end{proof}

\begin{lemma}\label{lem4.4}
The set $S$ is a scrambled set of the shift on $\mathbb{N}^{\mathbb{N}}$.
\end{lemma}

\begin{proof}
For any $u,v \in S\subset \mathbb{N}^{\mathbb{N}}$ with $u\neq v$, denoted by $u\in S_N$ and $v\in S_M$,
we shall prove that $(u,v)$ is a scrambled  pair for the shift.
The proof is divided into two cases according to $N=M$ or $N\neq M$.

(1). Assume that   $u,v\in S_N=\Delta_N(\Sigma_N)$ for some $N\geq 2$,
then there exist $x,y\in \Sigma_N$
with $x\neq y$ such that $u=\Delta_N(x)$ and $v=\Delta_N(y)$.
Since $x\neq y$, there exists $k\geq 1$ such that $x_k\neq y_k$.
By the definitions of $g_N$ and $\Theta_N$, the symbol $x_k$ and $y_k$ appear infinitely
often in the same location of $\Theta_N\circ g_N(x)$ and $\Theta_N\circ g_N(y)$ respectively.
By the definition of $\Delta_N$, hence  $x_k$ and $y_k$ appear infinitely
often in the same location of $\Delta_N(x)$ and $\Delta_N(y)$ respectively.
That is to say,
there  exists an increasing sequence $\{n_i\}_{i\geq 1}$ such that
\begin{center}
$(\sigma^{n_i}(u))_1= x_{k} \neq y_{k} =(\sigma^{n_i}(v))_1$,
\end{center}
 it follows that
$$\lim\limits_{i\rightarrow \infty}d(\sigma^{n_i}(u),\sigma^{n_i}(v))=1>0.$$
On the other hand,
according to the maps $g_N$ and $\Theta_N$, there exists an increasing sequence $\{m_j\}_{j\geq 1}$
such that $$(\Theta_N \circ g_N(x))|_{m_j+1}^{m_j+j}=(\Theta_N \circ g_N(y))|_{m_j+1}^{m_j+j}
=(\underbrace{1,1,\cdots, 1}_{j}),$$
by the definition of  $\Delta_N$, there exists an increasing sequence $\{l_j\}_{j\geq 1}$
such that $$u|_{l_j+1}^{l_j+j}=(\Delta_N(x))|_{l_j+1}^{l_j+j}=(\Theta_N\circ g_N(x))|_{m_j+1}^{m_j+j}
=(\underbrace{1,1,\cdots, 1}_{j})=(\Theta_N\circ g_N(y))|_{m_j+1}^{m_j+j}
=(\Delta_N(y))|_{l_j+1}^{l_j+j}=v|_{l_j+1}^{l_j+j},$$
that is, $d(\sigma^{l_j}(u),\sigma^{l_j}(v))\leq 2^{-j}$,
which implies
$$\lim\limits_{j\rightarrow \infty}d(\sigma^{l_j}(u),\sigma^{l_j}(v))=0.$$

(2). Suppose $u\in S_N$ and $v\in S_M$ such that $N,M\geq 2$ and $N\neq M$,
then there exist $x\in \Sigma_N$ and $y\in \Sigma_M$ such that $u=\Delta_N(x)$ and
$v=\Delta_M(y)$. By the definitions of $g_N$, $\Theta_N$, $g_M$ and $\Theta_M$, we have
\begin{center}
$(\Theta_N\circ g_N(x))_1=N$ and $(\Theta_M\circ g_M(y))_1=M$,
\end{center}
by the definition of  $\Delta_N$, there exists an increasing sequence $\{n_i\}_{i\geq 1}$ such that
\begin{center}
$u_{n_i+1}=(\Theta_N\circ g_N(x))_1=N$ and $v_{n_i+1}=(\Theta_N\circ g_N(y))_1=M$,
\end{center}
 it follows that
$$\lim\limits_{i\rightarrow \infty}d(\sigma^{n_i}(u),\sigma^{n_i}(v))=1>0.$$
On the other hand,
similarly to the case of $N=M$,
there exists an increasing sequence $\{l_j\}_{j\geq 1}$
such that $$u|_{l_j+1}^{l_j+j}=(\underbrace{1,1,\cdots, 1}_{j})=v|_{l_j+1}^{l_j+j},$$
that is, $d(\sigma^{l_j}(u),\sigma^{l_j}(v))\leq 2^{-j}$,
which implies
$$\lim\limits_{j\rightarrow \infty}d(\sigma^{l_j}(u),\sigma^{l_j}(v))=0.$$
\end{proof}

Since the CF expansion of $x\in [0,1)\cap \mathbb{Q}^{c}$ is unique, we can
define a continuous  bijection $\phi$ from $ \mathbb{N}^{\mathbb{N}}$ to
$ [0,1)\cap \mathbb{Q}^{c}$ by
\begin{center}
$\phi(a_1,a_2,a_3,\cdots)=[a_1,a_2,a_3,\cdots]$
\end{center}
for  $(a_1,a_2,a_3,\cdots)\in \mathbb{N}^{\mathbb{N}}$.
The following  diagram  is commutative, that is, $\phi\circ \sigma=T\circ \phi$.
\[
\begin{CD}
\mathbb{N}^{\mathbb{N}} @>\sigma>>  \mathbb{N}^{\mathbb{N}}\\
@V{\phi}VV             @VV{\phi}V\\
[0,1)\cap \mathbb{Q}^{c}  @>T>>   [0,1)\cap \mathbb{Q}^{c}
\end{CD}
\]

The following lemma implies that the set $\phi(S)$ is a scrambled set of $T$ on $[0,1)$.

\begin{lemma}\label{lem10}
The set $\phi(S)$ is a scrambled set of $T$ on $[0,1)$.
\end{lemma}
\begin{proof}
For any $u,v\in S$ with $u\neq v$,  denoted by $u\in S_N$ and $v\in S_M$,
we shall prove that $(\phi(u),\phi(v))$ is a scrambled  pair for $T$.

(1). $\mathbf{Lower~~ limits}$.  According to Lemma \ref{lem4.4},
$\liminf\limits_{n\rightarrow \infty}d(\sigma^{n}(u),\sigma^{n}(v))=0,$
there exists an increasing sequence $\{n_j\}_{j\geq 1}$
such that $$\lim\limits_{j\rightarrow \infty}d(\sigma^{n_j}(u),\sigma^{n_j}(v))=0.$$
Since the map $\phi$ is continuous and $\phi\circ \sigma=T\circ \phi$, we obtain
$$\lim\limits_{j\rightarrow \infty}|T^{n_j}\circ \phi(u)-T^{n_j}\circ \phi(v)|
=\lim\limits_{j\rightarrow \infty}|\phi(\sigma^{n_j}(u))-\phi(\sigma^{n_j}(v))|=0.$$

(2). $\mathbf{Upper~~ limits}$.
The proof is divided into two cases according to $N=M$ or $N\neq M$.

$\rm(i)$ Suppose that $N=M$, that is,  $u, v\in S_N$ for some $N\geq 2$,
then $u_i,v_i\in\{1,2,\cdots, N\}$ for any $i\geq 1$.
According to Lemma \ref{lem4.4},
 $\limsup\limits_{n\rightarrow \infty}d(\sigma^{n}(u),\sigma^{n}(v))=1,$
there exists an increasing sequence $\{m_i\}_{i\geq 1}$ such that
$(\sigma^{m_i}(u))_1 \neq  (\sigma^{m_i}(v))_1$. Then
\begin{eqnarray*}
\begin{split}
|T^{m_i}\circ \phi(u)-T^{m_i}\circ \phi(v)| &= |\phi(\sigma^{m_i}(u))-\phi(\sigma^{m_i}(v))|\\
&\geq \min\{|I((\sigma^{m_i}(u))_1,(N+1) )|, |I((\sigma^{m_i}(v))_1,(N+1) )|\}\\
&\geq \min\{\lambda^{-1}|I((\sigma^{m_i}(u))_1)||I(N+1)|, \lambda^{-1}|I((\sigma^{m_i}(v))_1)||I(N+1)|\}\\
&\geq \lambda^{-1}|I(N+1)|^{2}>0
\end{split}
\end{eqnarray*}
where the first inequality holds
since  the interval $I((\sigma^{m_i}(u))_1,(N+1) )$ or $I((\sigma^{m_i}(v))_1,(N+1) )$ is contained in the gap
between the point $\phi(\sigma^{m_i}(u))$ and $\phi(\sigma^{m_i}(v))$
and $u_i,v_i\in\{1,2,\cdots, N\}$ for any $i\geq 1$
and the second inequality holds by  Lemma \ref{lem4}.
Therefore
$$\limsup\limits_{n\rightarrow \infty}|T^{n}\circ \phi(u)-T^{n}\circ \phi(v)|>0.$$

$\rm(ii)$ Suppose that $u\in S_N$, $v\in S_M$ for some $N,M\geq 2$ and $N\neq M$.
According to (2) of Remark \ref{rem4.1} and constructions of $S_N$ and $S_M$,
there exists an increasing sequence $\{l_i\}_{i\geq 1}$ such that
$$(\sigma^{l_i}(u))_1(\sigma^{l_i}(u))_2=N1,~~  (\sigma^{l_i}(v))_1(\sigma^{l_i}(v))_2=M1,$$
Then
\begin{eqnarray*}
\begin{split}
|T^{l_i}\circ \phi(u)-T^{l_i}\circ \phi(v)| &= |\phi(\sigma^{l_i}(u))-\phi(\sigma^{l_i}(v))|\\
&\geq \min\{|I(N,2)|,|I(M,2)|\}\\
&\geq \lambda^{-1}|I(2)| \min\{|I(N)|,|I(M)|\}>0
\end{split}
\end{eqnarray*}
where the first inequality holds since the interval $I(N,2)$ or $I(M,2)$ is contained in the gap
between the point $\phi(\sigma^{l_i}(u))$ and $\phi(\sigma^{l_i}(v))$ and
the second inequality holds by  Lemma \ref{lem4}.
Therefore
$$\limsup\limits_{n\rightarrow \infty}|T^{n}\circ \phi(x)-T^{n}\circ \phi(y)|>0.$$
\end{proof}

\subsection{Estimation of  $\mathbf{\dim_H\phi(S_N)}$ }

Recall that $E_N=\{x\in[0,1):~~1\leq a_n(x)\leq N, ~~\forall n\geq 1 \}$ for any $N\geq 2$.
Since the Hausdorff dimension  of rational number is zero, we only need to consider
the set $E_N\cap \mathbb{Q}^{c}$. Note that $\phi(\{1,2,\cdots, N\}^{\mathbb{N}})=E_N\cap \mathbb{Q}^{c}$.
By the construction of $S_N$,  $S_N\subset \{1,2,\cdots, N\}^{\mathbb{N}}$,
we know that $\phi(S_N)\subset E_N\cap \mathbb{Q}^{c}$ for any $N\geq 2$, then
$$\dim_H\phi(S_N)\leq \dim_H(E_N\cap \mathbb{Q}^{c})=\dim_HE_N.$$

\begin{proposition}\label{prop4.1}
For any $N\geq 2$, we have $\dim_H\phi(S_N)=\dim_HE_N$.
\end{proposition}

Consider a map $g:\phi(S_N)\rightarrow E_N\cap \mathbb{Q}^{c}$ defined by
$$ g(b)=\phi\circ \Delta_N^{-1} \circ \phi^{-1}(b)$$
for $b\in \phi(S_N)$. Since $\phi$ is continuous and bijective from $\mathbb{N}^\mathbb{N}$
to  $[0,1)\cap \mathbb{Q}^{c}$ and $\Delta_N$ is  continuous and bijective  from  $\Sigma_N$
to $S_N=\Delta_N(\Sigma_N)$,
the map $g$ is a continuous bijection on $\phi(S_N)$.
Proposition \ref{prop4.1} is the corollary of the following lemma, thus we omit the proof of Proposition \ref{prop4.1}.
\begin{lemma}\label{lem4.6}
For any $\epsilon>0$, the map $g$ satisfies locally $\frac{1}{1+\epsilon}$-H$\ddot{o}$lder condition.
\end{lemma}

Recall that the map $f:X\rightarrow \mathbb{R}$ $(X\subset \mathbb{R})$ satisfies
locally $\alpha$-H$\rm \ddot{o}$lder condition, if there exist a real number $r>0$ and a constant $C>0$ such that,
for any $x,y\in X$ with $|x-y|<r$,
$$|f(x)-f(y)|\leq C |x-y|^{\alpha}.$$

Before proving Lemma \ref{lem4.6},
we make use of a kind of symbolic space described as follows:

For a fixed integer $N\geq 2$ and any $n\geq 1$, we can check that
$$\phi(S_N)=\bigcap\limits_{n\geq 1}\bigcup\limits_{(b_1,b_2,\cdots, b_n)\in A_n}I(b_1,b_2,\cdots, b_n),$$
where
$$A_n=\{(b_1,b_2,\cdots, b_n)\in \{1,2,\cdots,N\}^{n}:
(b_1,b_2,\cdots, b_n)=(x_1,x_2,\cdots, x_n), \forall (x_1,x_2,\cdots)\in S_N\}.$$
Recall that $R=\{r_n\}_{n\geq 1}:=\bigcup\limits_{m=1}^{\infty}\bigcup\limits_{t=1}^{m}\{k:k=m^3+t\}$.
For any $n\geq 1$, let $t(n)=\sharp\{k\geq 1: r_k\leq n, r_k\in R\}$, we obtain
$$t(n)\leq 1+2+\cdots+(n^{\frac{1}{3}}+1) \leq
\dfrac{(n^{\frac{1}{3}}+1)(n^{\frac{1}{3}}+2)}{2}\leq 2n^{\frac{2}{3}},~~ when~~n\geq 8,$$
similarly,
$$t(n)\geq 1+2+\cdots+(n^{\frac{1}{3}}-2) \geq
\dfrac{(n^{\frac{1}{3}}-2)(n^{\frac{1}{3}}-1)}{2}\geq  \frac{n^{\frac{2}{3}}}{8},~~ when~~n\geq 64.$$
For any $(b_1,b_2,\cdots, b_n)\in A_n$, let $\overline{(b_1,b_2,\cdots, b_n)}$ be the finite word
by eliminating the terms $\{b_{r_k}: r_k\leq n, r_k\in R \}$.
Then
 $$\overline{(b_1,b_2,\cdots, b_n)}\in \{1,2,\cdots,N\}^{n-t(n)}. $$

For convenience, set
$$\overline{q_n}(b_1,b_2,\cdots, b_n)=q_{n-t(n)}\overline{(b_1,b_2,\cdots, b_n)}~~~~~~
and ~~~~~~ \overline{I}(b_1,b_2,\cdots, b_n)=I\overline{(b_1,b_2,\cdots, b_n)}.$$

For any $b=[b_1,b_2,\cdots],c=[c_1,c_2,\cdots]\in\phi(S_N)\subset [0,1)$ and $b\neq c$,
there exists  $n\in \mathbb{N}$ such that
$(b_1,\cdots, b_n)=(c_1,\cdots, c_n)$ and $b_{n+1}\neq c_{n+1}$.

Based on the following fact: for any $w\in \mathbb{N}^{*}$, the subintervals
 $I(w,i)~~(i\in \mathbb{N})$  are positioned as follows (see FIGURE \ref{Fig.1}).
\begin{figure}[H]
  \centering
  \includegraphics[width=5.5in]{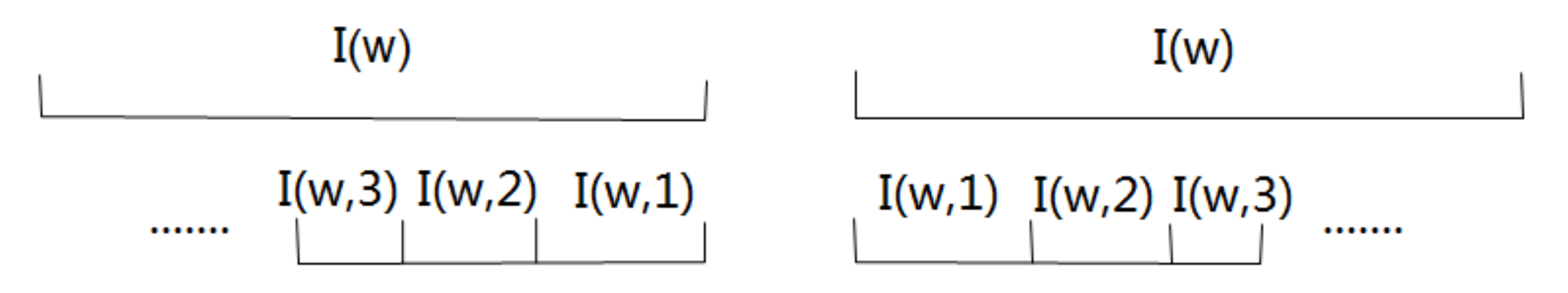}\\
  \caption{Subintervals $I(w,i)$ of $I(w)$ for the Gauss map  }\label{Fig.1}
\end{figure}
\begin{lemma}\label{lem8}
We have $|b-c|\geq \lambda^{-2}|I(N+1)|^2  |I(b_1,\cdots, b_{n})|$
where $\lambda$ is defined in Lemma $\rm \ref{lem4}$.
\end{lemma}
\begin{proof}
Suppose that $b_{n+1}< c_{n+1}$.
The interval $I(c_1,\cdots, c_{n},c_{n+1},(N+1))$ is contained
in the gap between $b$ and $c$ whether $n$ is even or odd.
Similarly to the case $c_{n+1}< b_{n+1}$, the interval $I(b_1,\cdots, b_{n}, b_{n+1},(N+1))$ is contained
in the gap between $b$ and $c$ whether $n$ is even or odd.
By Lemma \ref{lem4} and $1 \leq b_{n+1}\neq c_{n+1}\leq N$,
\begin{eqnarray*}
\begin{split}
|b-c| &\geq \min\{|I(b_1,\cdots,b_{n}, b_{n+1},N+1)|, |I(c_1,\cdots, c_{n}, c_{n+1},N+1)|\}\\
&\geq  \min\{\lambda^{-2}|I(b_1,\cdots,b_{n})|\cdot|I(b_{n+1})|\cdot |I(N+1)|,~~
 \lambda^{-2}|I(c_1,\cdots, c_{n})|\cdot|I(c_{n+1})|\cdot |I(N+1)| \}\\
&\geq \lambda^{-2} |I(b_1,\cdots,b_{n})| |I(N+1)|^2
\end{split}
\end{eqnarray*}
\end{proof}

\noindent{\bf Proof of Lemma \ref{lem4.6}:}
From Lemma \ref{lem0}, for any $(b_1,b_2,\cdots, b_n)\in A_n$, we have
$$\overline{q_n}^{2}(b_1,b_2,\cdots, b_n)=q_{n-t(n)}^{2}\overline{(b_1,b_2,\cdots, b_n)}
\geq 2^{(n-t(n)-1)}.$$
Let $\epsilon>0$,  there exists $K=K(\epsilon)>64$ such that for any $n\geq K$, we have
$$2^{(n-t(n)-1)\epsilon}\geq 2\cdot (N+1)^{2t(n)}.$$
By Lemma \ref{lem0} and  Lemma \ref{lem01},
\begin{eqnarray*}
\begin{split}
|I(b_1,b_2,\cdots, b_n)| &\geq \frac{1}{2q_n^2(b_1,b_2,\cdots, b_n)}
\geq \frac{1}{2q_{n-t(n)}^2\overline{(b_1,b_2,\cdots, b_n)}(N+1)^{2t(n)}}\\
&\geq\frac{1}{q_{n-t(n)}^{2(1+\epsilon)}\overline{(b_1,b_2,\cdots, b_n)}}
\geq |\overline{I}(b_1,b_2,\cdots, b_n)|^{1+\epsilon}
\end{split}
\end{eqnarray*}
Let $r<\lambda^{-2}|I(N+1)|^2 \cdot \min\limits_{(b_1,\cdots, b_{K})\in A_{K}} |I(b_1,\cdots, b_{K})|$.
For any $c\in (b-r,b+r)$, there exists $n$ such that $(b_1,\cdots, b_n)=(c_1,\cdots, c_n)$ and $b_{n+1}\neq c_{n+1}$,
where $b=[b_1,b_2,\cdots],c=[c_1,c_2,\cdots]\in\phi(S_N)\subset [0,1)$, then
\begin{eqnarray*}
\begin{split}
|g(b)-g(c)|& = |\phi\circ \Delta_N^{-1} \circ \phi^{-1}(b)-\phi\circ \Delta_N^{-1} \circ \phi^{-1}(c)|\\
&\leq |\overline{I}(b_1,b_2,\cdots, b_n)|
\leq|I(b_1,b_2,\cdots, b_n)|^{\frac{1}{1+\epsilon}}\\
&\leq(\lambda^{2}|I(N+1)|^{-2})^{\frac{1}{1+\epsilon}}|b-c|^{\frac{1}{1+\epsilon}}
\end{split}
\end{eqnarray*}
where the last inequality holds by Lemma \ref{lem8}.
\qed

\vskip 0.2cm
\noindent{\bf Proof of Theorem \ref{thm2}:}
 By Lemma \ref{lem10}, the set $\phi(S)$ is a scrambled set of $T$ on $[0,1)$
 where $S=\bigcup\limits_{N\geq 2}S_N$.
 Since $\dim_H \phi(S)\geq \dim_H \phi(S_N)= \dim_H E_N\geq 1-\dfrac{1}{N\log2}$
 for any $N\geq 2$, thus $\dim_H \phi(S)=1$.
\qed

{\bf Acknowledgement.} The authors thank Professor Ying Xiong for the helpful suggestions. The work was supported by NSFC 11371148, Guangdong
Natural Science Foundation 2014A030313230, and "Fundamental Research Funds
for the Central Universities" SCUT 2015ZZ055 and 2015ZZ127.

\bibliographystyle{amsplain}

\end{document}